\documentclass{article}

\usepackage{amsmath,amssymb,amsbsy,amsfonts,amsthm}
\usepackage{graphicx}
\usepackage{hyphenat}

\theoremstyle{definition}
\newtheorem{definition}{Definition}
\newtheorem*{definition*}{Definition}
\newtheorem{lemma}{Lemma}
\newtheorem*{lemma*}{Lemma}
\newtheorem{fact}{Fact}
\newtheorem*{fact*}{Fact}
\newtheorem{theorem}{Theorem}
\newtheorem*{theorem*}{Theorem}
\newtheorem{corollary}{Corollary}
\newtheorem*{corollary*}{Corollary}

\newtheorem*{property*}{Property}

\newtheorem{remark}{Remark}
\newtheorem*{remark*}{Remark}

\newtheorem*{proposition*}{Proposition}

\newcommand{\bb}{\mathbb}
\newcommand{\mc}{\mathcal}
\newcommand{\Var}{\text{Var}}

\newcommand{\ra}{\rightarrow}
\newcommand{\inte}{\text{int}}

\title{A Central Limit Theorem for First Passage Percolation in the Slab}

\author{Serena Sian Yuan}
\date{}

\begin{document}

\maketitle

\begin{abstract}
We consider first-passage percolation on the edges of $\mathbb{Z}^2 \times \{1, \cdots, k\},$ namely the slab $\mathbb{S}_k$ of width $k$. Each edge is assigned independently a passage time of either 0 (with probability $p_c(\mathbb{S}_k)$) or 1 ( with probability $1-p_c(\mathbb{S}_k)$) where $p_c$ is the critical probability for Bernouilli percolation. We prove central limit theorems for point-to-point and point-to-line passage times. These generalize the results of Kesten and Zhang \cite{kestenzhang} to non-planar graphs.
\end{abstract}

\section{Introduction}

\; \; Percolation is a fundamental discrete model in random spatial processes and statistical physics that manifest phase transitions. 
Bond percolation is the usual model on $\mathbb{Z}^d$, which has a phase transition at the 
critical point $p_c \in (0, 1)$, the connected components are almost surely finite for 
$p < p_c$ and for $p > p_c$ there is a unique infinite component.

Given $G= (V(G), E(G))$ with edge weights $(t_e)_{e \in G}$ and a path $\pi \subset E(G)$, define the passage time 
$$
T(\pi) = \sum_{e \in \pi} t_e.
$$
Throughout the paper, we take $t_e = 0$ or $t_e=1$.

For a path $\pi \subset G = \bb{Z}^d$ from $x$ to $y$, the first passage time between two vertices is given by the following,

 $$
 T(x, y) = \inf \{ \sum T(\pi) : \pi \text{ is a path from } x \text{ to } y \}.
 $$ 

The passage time between two vertex sets $A, B \subset V(G)$ is defined as

\begin{equation}
T(A,B) = \inf \{ T(\pi) : \pi \text{ is a path connecting some vertex of A with some vertex of B } \}.
\end{equation}

Let the variable $a(0,n) = T(0,n \mathbf{e_1})$ be called the \emph{point to point passage time} where $\mathbf{e_1}$ is the first coordinate vector. 

Let the variable $b(0,n) = T(0, H_n )$ be the \emph{point to line passage time} where $$H_n : = \{ x \in G : x \cdot e_1 \geq n \}.$$  

 A generalization of $a(0,n)$ is given by $T(0,n \mathbf{u})$, the passage time from the origin $\{ 0 \}$ to the nearest point on the graph $G$ to $n \mathbf{u}$ for given unit vector $\mathbf{u}$.

We work with first passage percolation with $\{ t(e) : e \in E(G)\}$ i.i.d. Bernoulli random variables, with the probability given by,

\begin{equation}
\label{eqn:pre1}
\bb{P}( t(e)=1) = F(1) 
\end{equation}

and 

\begin{equation}
\label{eqn:pre2}
\bb{P}( t(e)=0) = F(0)
\end{equation}

We let $t=0$ be equivalent to closed edges, and let $t=1$ be equivalent to open edges. 

Critical first passage percolation occurs when $F(0)= p_c(G)$, the critical probability of bond percolation of $G$. 

The following dichotomy is well-known as one changes the distribution of $\{t_e\}$. For $F(0) <p_c$, $\frac{a_0}{n}$ and $\frac{b_0}{n}$ converge almost surely to a strictly positive constant. For $F(0) > p_c$, the families of random variables $\{ a_{0,n} \}$ and $\{b_{0,n} \}$ are tight (See e.g. [\cite{Kesten86} Theorem 6.1] and \cite{Zhangzhang}). 


When $F(0) = p_c$, in \cite{Kesten86} it is proved for $\mathbb{Z}^2$ that for each unit vector $u$,
$$
C_3 \log n \leq \bb{E} T(0, nu) \leq C_4 \log n.
$$


\subsection{Main Result}
 
Our main result shows a CLT holds on slabs $\bb{S}_k = \bb{Z}^2 \times \{0, 1, \cdots, k\}$. We will prove the theorem for $b_{0,n}$, but the same proof also works for $a_{0,n}$ and for general $T(0,n\mathbf{u})$. 
 
  \begin{theorem}
  \label{thm:1}
We simply denote $n =(n,0,0)$ . Given that $F$ is defined by (\ref{eqn:pre1}) and (\ref{eqn:pre2}) with $F(0) = p_c(\mathbb{S}_k)$, there exist constants $ 0 < C_1 , C_2 < \infty$ such that
\begin{equation}
\label{eqn:1.7}
  C_1 \log n \leq \Var T(0,n) \leq C_2 \log n, n \geq 2.
\end{equation}
  
Moreover,

\begin{equation}
\label{eqn:1.8}
\frac{b_{0,n} - \bb{E} b_{0,n} }{ \sqrt{ \Var \; T(0,n) } } \overset{d}{\rightarrow} N(0,1)
\end{equation}

and for any $u \in S^1$

\begin{equation}
\label{eqn:1.9}
\frac{T(0,n\mathbf{u}) - \bb{E} T(0,n\mathbf{u})}{\sqrt{ \Var \; T(0,n\mathbf{u})}} \overset{d}{\rightarrow} N(0,1)
\end{equation}
  \end{theorem}

We prove Theorem \ref{thm:1} by representing $b_{0,n} - \bb{E} b_{0,n}$ as a sum of martingale differences, and this would allow us to apply a central limit theorem for martingales, by McLeish \cite{McLeish}.

We give a brief summary of the ideas in the proof for our result.

\subsection{Outline of Proof.}


Two main ingredients for our proof are an adaptation of the martingale central limit theorem in \cite{McLeish} and the Russo Seymour Welsh Theorem for Bernouilli percolation on slabs \cite{wu}. The analogous result on $\mathbb{Z}^2$ was proved by Kesten and Zhang \cite{kestenzhang} by adapting the martingale central limit theorem and the RSW Theorem on $\mathbb{Z}^2$, and considering the passage time within annuli $A_{n, 2n}$. We now summarize the argument of Kesten and Zhang for $\mathbb{Z}^2$.

Let $S(n)$ denote the ball of a given radius $n$. Let $A(r_1,r_2)$ denote the annulus with inner radius $r_1$ and outer radius $r_2$. By the Russo-Seymour-Welsh Theorem, with probability uniformly bounded from below, there is a $p_c$-closed circuit in $A(2^p, 2^{p+1})$. If it exists, let $\mathcal{C}_p$ denote the innermost (with respect to lexicographical ordering) circuit in $A_{2^p, 2^{p+1}}.$


Note any two vertices $v', v''$ on $\mc{C}_p$ are connected by a path that is part of $\mc{C}_p$ and have passage time equal to zero (since $\mc{C}_p$ is closed, and summing edges where $t_e=0$ gives no contribution to the passage time). Therefore, for all vertices $v \in \mc{C}_p$, the values of $T(\mathbf{0},v)$ are the same.

This fact implies that we may sum the passage times as,
\begin{equation}
\label{eqn:7}
T(\mathbf{0}, \mc{C}_q) = \sum_{p=0}^q T(\mc{C}_{p-1}, \mc{C}_p).
\end{equation}

Since $b_{0,n}$ is well-approximated by $T(0,\mc{C}_q)$ for $q$ that satisfies $2^{q-1} <n < 2^q$, it suffices to prove a CLT for (\ref{eqn:7}).

Given a circuit $C$ surrounding the origin and lying outside of $S(2^p)$, the event 

$$
\{ \mc{C}_p =C \}
$$

only depends on 
\begin{align*}
t(e) \text{ s.t. } e \in C \cup \text{int}(C) \setminus S(2^p), \text{ where int}(C)\text{ is the interior of } C.
\end{align*}

We may clearly see that random variables $\{ T(\mc{C}_{p-1}, \mc{C}_p), \mc{C}_p \}_{p \geq 0}$ form a Markov chain, since it is possible to determine $\mc{C}_p$ once $\mc{C}_{p-1}$ is given, even without knowledge of values $t(e)$ for any edges $e \in \text{int}(\mc{C}_{p-1})$. 

Therefore, the proof for the CLT relies on the sum of martingale differences representation of $b_{0,n} - \bb{E} b_{0,n}$.

Define

\begin{align}
\label{eqn:1.21}
\mc{F}_p = \sigma-\text{field generated by } \mathcal{C}_p \text{ and } \{ t_e | e  \in \inte( \mc{C}_p) \} 
\end{align}

We therefore have,
\begin{equation}
\label{eqn:b0n}
b_{0,n} - \bb{E} b_{0,n} = \sum_{p=0}^q (\bb{E} [b_{0,n} | \mc{F}_p ] - \bb{E} [b_{0,n} | \mc{F}_{p-1}] ) .
\end{equation}

Then $\mc{G}_p := \bb{E} [b_{0,n} | \mc{F}_p] - \bb{E} [b_{0,n} | \mc{F}_{p-1} ] $ are martingale differences and are related to $T(\mc{C}_{p-1}, \mc{C}_p)$. The truncated versions of $\mc{G}_{p_1}$ and $\mc{G}_{p_2}$ are nearly independent for $| p_1 - p_2|$ large. This allows us to apply a central limit theorem for martingales (\cite{McLeish}) to obtain (\ref{eqn:1.8}).

\begin{figure}
\label{fig:widthkcircuit}
\centering
\includegraphics[width=3cm]{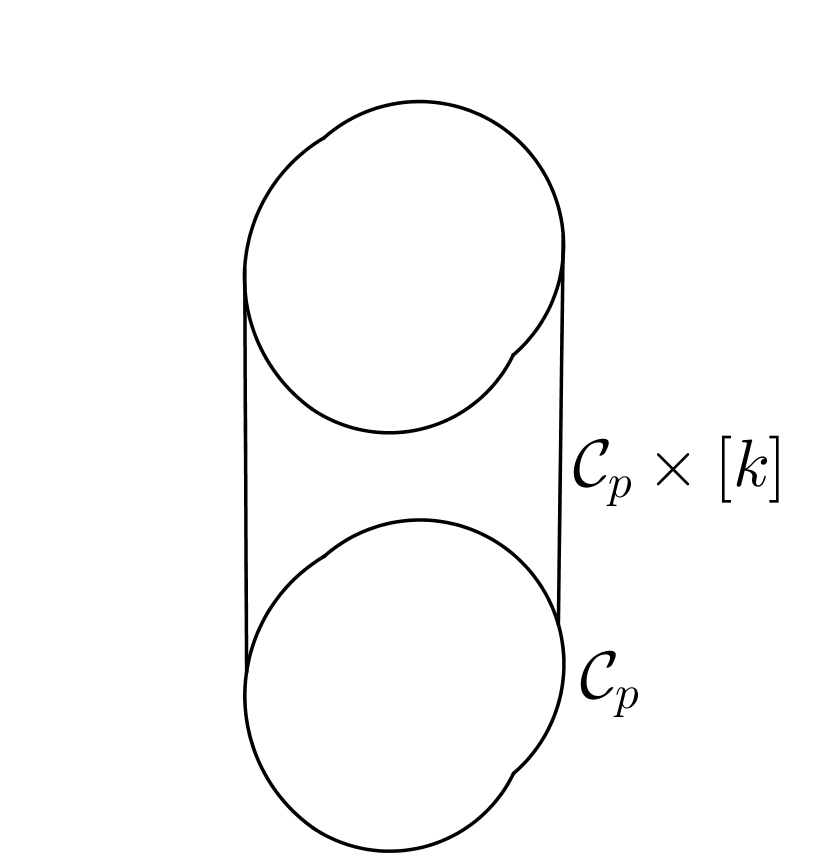}
\caption{$\overline{\mc{C}_p}$}
\end{figure}

We now discuss necessary modifications to prove Theorem \ref{thm:1} for $\mathbb{S}_k$. Given $A \subset \mathbb{Z}^2$, denote by
\begin{align*}
\overline{A} = A \times \{ 0, \cdots, k\}. 
\end{align*}

Applying the RSW Theorem proved in \cite{wu}, we see that with probability uniformly bounded from below, there is a $p_c$-closed circuit in $\overline{A}(2^p, 2^{p+1})$. Still denote by $\mc{C}_p$ the innermost such circuit (given a fixed ordering of edges).

However, the equality (\ref{eqn:7}) would fail since the geodesic from $0$ to $\mc{C}_q$ may not intersect the $\{ \mc{C}_p \}_{p <q}$ (they will intersect $\{ \overline{\mc{C}}_p\}_{p < q}$ though). To deal with this, we modify the Def. (\ref{eqn:1.21}) as 
\begin{align*}
\mc{F}_p = \sigma \text{-field generated by } \overline{\mc{C}_p} \text{ and } \{t_e : e \in \text{int}(\overline{\mc{C}_p}) \}.
\end{align*}

Then we can still write $b_{0,n} - \bb{E} b_{0,n}$ as the martingale difference sum $b_{0,n} - \bb{E} b_{0,n} = \sum_{p \leq q} \triangle_p$, where $\triangle_p = \mathbb{E}[b_{0,n} | \mc{F}_p] - \mathbb{E}[b_{0,n} | \mc{F}_{p-1} ]$. 

The increments $( \triangle_p )_{p \leq q}$ will be different from the $(\mc{C}_p)_{p \leq q}$, but only by a finite number (in fact it is bounded by $\mc{C}_k$, where $k$ is the width of the slab). Therefore it still satisfies the condition of McLeish's CLT \cite{McLeish} and an application of that CLT concludes Theorem \ref{thm:1}. 

\section{Preliminary Results}

In this section we recall the RSW Theorem for critical Bernouilli percolation on $\mathbb{S}_k$,  proven in [\cite{wu} Theorem 3.1].

Let $\bb{P}_p$ be the product measure on the configuration space $\{0,1\}^E = \Omega$, such that $\bb{P}_p(e=1)=p$. 
 
For $x < x'$, denote $[x, x'] = \{ x, x+1, \cdots, x' \}$. Let $R = \overline{[x, x'] \times [y, y']}$ be a rectangle in $\mathbb{S}_k$ that is equivalent to $[x, x'] \times [y, y'] \times [k]$. Say $R$ is crossed horizontally (denoted as $\mc{H}(R)$) if there is an open path from $\overline{\{x\} \times [y, y']}$ to $\overline{\{x'\} \times [y,y']}$ inside $R$. 

For $m,n \geq 1$, and for $p \in [0,1]$, let us define

$$
f(m, n) = f_p (m,n) := \bb{P}_p [ \mc{H} (\overline{[0,m] \times [0,n]})].
$$

\begin{theorem}{(Box-crossing property)}
\label{thm:BCP}
Let $p= p_c(\mathbb{S}_k).$ For $\rho>0$, there exists a constant $c_{\rho} \in (0, 1)$, independent of $n$, such that for every $n \geq 1/ \rho$,
\begin{equation}
c_{\rho} \leq f(n, \lfloor \rho n \rfloor) \leq 1 - c_{\rho}.
\end{equation}
\end{theorem}

\begin{remark}
The constant $c_{\rho}$ depends on the thickness $k$ of the slab. 
\end{remark}

The following are some consequences of the box-crossing property on the slab as given in [\cite{wu} Corollary 3.2].

\begin{corollary}
\label{cor:1}

For critical percolation on $\mathbb{S}_k$, we have
\begin{enumerate}
\item \label{item:1} ( Existence with positive probability of circuits in the annulus $\Lambda_{2n} \setminus \Lambda_n$.) 

There exists $c>0$ such that for every $n \geq 1$,
\begin{equation}
\mathbb{P}_p [\text{there exists an open circuit in } \overline{A_{n, 2n}} \text{ surrounding } \bar{B_n}] \geq c. 
\end{equation}

\item \label{item:2} (Existence of blocking surfaces with positive probability.) 

There exists $c>0$ such that for every $n \geq 1,$
\begin{equation}
\mathbb{P}_p [\text{there exists an open path from } \bar{B_n} \text{ to } \partial \bar{B_{2n}}] \leq 1 - c. 
\end{equation}



\end{enumerate}

\end{corollary}

\section{Proof of Theorem} 

We consider the dyadic scales,
\begin{equation}
\label{eqn:2.1}
 \{ 2^q \}_{q \in \bb{N}}
\end{equation}

Let $S(n) = [-n, n]^2 \times [k] $ be the square of size $2n$ with width $k$ centered at the origin, and $\partial S(n) = \partial [-n,n]^2 \times [k] $. 

Let the annulus between scales $S(2^p)$ and $S(2^{p+1})$ be defined as 

\begin{equation}
\label{eqn:annulusdefinition}
A(p) = S(2^{p+1}) \setminus S(2^p)
\end{equation}


We define $m(p)$ for $p \geq 0$ as

\begin{equation}
\label{eqn:defm}
m(p) = \inf \{ t \in \{ p, p+1, \cdots,\} : A(t) \text{ contains an open circuit surrounding the origin } \}.
\end{equation}

Properties of $m(p)$ include:
\begin{remark}
$m(p) \geq p$ but it is possible for $m(p) = m(p') \geq p' >p$, which occurs when there is no dual closed surface surrounding the origin in any of the annuli $A(p), A(p+1), \cdots, A(p'-1)$.
\end{remark}

Define the innermost open circuit as
\begin{equation}
\label{eqn:2.6}
\text{For } p \geq 0,
\mc{C}_p = \{ \text{ innermost $p_c$-open circuit that surrounds the origin } \textbf{0} \text{ within } A(m(p)) \}.
\end{equation}

We have the following fact.

\begin{fact}
\label{fact:2.7}
By definition, $p_1 \leq p_2$ implies $m(p_1) \leq m(p_2)$. Therefore, either 
\begin{itemize}
\item $m(p_1) = m(p_2)$ and $\mc{C}_{p_1} = \mc{C}_{p_2}$ or,
\item $m(p_1) < m(p_2)$ and $\mc{C}_{p_1} \subset A(m(p_1)) \subset S(m(p_2)) \subset \text{int}(\mc{C}_{p_2})$
\end{itemize}
must hold.
\end{fact}





We introduce

$$
\overline{\mc{C}_p } =\mc{C}_p \times \{0, \cdots, k \}.
$$




Recall from Equation (\ref{eqn:1.21})

$$\mc{F}_p = \sigma-\text{field generated by } \mathcal{C}_p \text{ and } \{ t_e | e  \in \inte( \overline{\mc{C}_p}) \}.$$ 

When considering the martingale difference array, instead of studying 


$$T(0, n) - \bb{E} T(0,n),$$

we study 

\begin{align}
\label{eqn:martingale}
  T(0, \mc{C}_{\ell}) - \bb{E} T(0, \mc{C}_{\ell}) := \sum_{p=1}^{\ell} \triangle_p.
\end{align}

where $\ell$ satisfies $2^{\ell-1} < n <2^{\ell}$ and where 




\begin{align}
\triangle_p ( \omega) = \bb{E} [ T(0, \mc{C}_{\ell}) |\mc{F}_p] - \bb{E} [ T(0,\mc{C}_{\ell}) | \mc{F}_{p-1}].
\end{align}


We now summarize some basic facts used in the proof.

\begin{fact}
\label{fact:1}
Given that there is a minimizing path from one point to the other, any subpath is a minimizer (between its extreme points). 
\end{fact}

Fact \ref{fact:1} follows easily from the triangle inequality.

We have the following fact about $\overline{C}_p$,

\begin{fact}
\label{fact:3}
$$|T(0,x) - T(0, \mc{C}_p) | \leq k, \; \forall x \in \overline{\mc{C}_p}.$$
\end{fact}


\begin{figure}
\centering
\label{fig:proj}
\includegraphics[width=5.75cm]{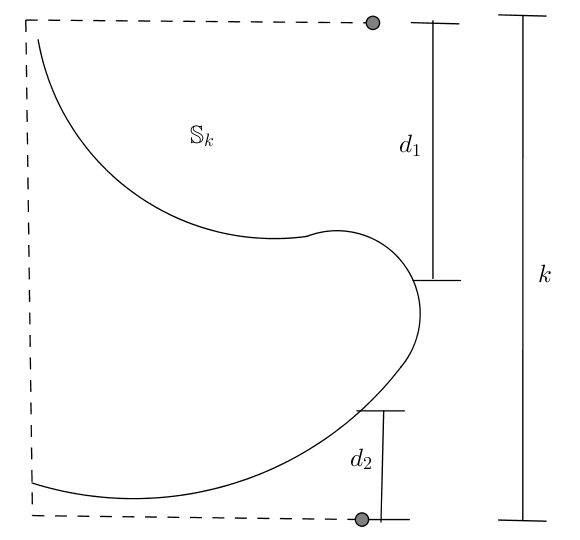}
\caption{$d_1$ and $d_2$ are bounded by $k$}
\end{figure}

Let $\omega$ and $\omega'$ be independent samples from $\Omega$. In the proofs we will fix a scale $p \in \bb{N}$ and use $\omega'(S(2^p)^c)$ to denote the edge configuration of $\omega'$ outside $S(2^p)$ and use $\omega(S(2^p))$ to denote the configuration of $\omega$ inside of $S(2^p)$. Also let $\bb{E}$ and $\bb{E}'$ be the expectation with respect to $\omega$ and $\omega'$. This allows us to study how the martingale difference $\triangle_p$ depends on the configurations $\omega$ inside $\mc{C}_p$.

Let $\tau_{\overline{\mathcal{C}_p(\omega)}}(\omega)$ be the first intersection point of the geodesic from $\mathbf{0}$ to $\mathcal{C}_{\ell}$ at $\overline{\mathcal{C}_p}$.

More precisely, we take an arbitrary ordering on all edges, and then order all the paths in lexicographical order. If there are more than one geodesic from $\mathbf{0}$ to $C_{\ell}$, choose the smallest path with respect to the ordering. Let $\{ \mc{U}_1, \mc{U}_2 , \cdots \} $ be this smallest path (it has to be self-avoiding). We let $\tau_{\overline{\mathcal{C}_p(\omega)}}(\omega) = \mc{U}_m$ such that $m = \inf \{ \ell: \mc{U}_{\ell} \in \overline{\mc{C}_p} \}.$ 

\begin{lemma}
\label{lem:1}


\begin{align}
\label{eqn:lem1final}
	\triangle_p = T(0, \mc{C}_p) - T(0, \mc{C}_{p-1})  + \bb{E}' [T(\tau_{\overline{\mc{C}_p}}, \mc{C}_{\ell})] - \bb{E}' [T(\tau_{\overline{\mc{C}_{p-1}}}, \mc{C}_{\ell})] + R
	\end{align}

    where $|R| \leq 8k$ with probability $1$. 
    
\end{lemma}

\begin{proof}

Let us fix configuration of edges $\omega$.
Any path $\pi$ that traverses from $0$ to $\overline{\mc{C}_{\ell}}$ must intersect $\overline{\mc{C}_{p-1}}$ and $\overline{\mc{C}_p}$. 

Let $\pi$ be the geodesic from 0 to $\overline{\mc{C}_{\ell}}$  on $\bb{Z}^2 \times [k]$ and let $\pi_1$ be the part of $\pi$ from $0$ to its first intersection with $\overline{\mc{C}_p}$.

 Let $\pi_2$ be the part of $\pi$ from its first intersection with $\overline{\mc{C}_p}$ to its first intersection with $\overline{\mc{C}_{\ell}}$.

We claim

\begin{align}
\label{eqn:A}
T(\pi_1) = T(0, \overline{\mc{C}_p})+ R_1,
\end{align}

where $|R_1| \leq 2k$ holds with probability $1$.  

Obviously, we have $T(\pi_1) \geq T(0, \overline{\mc{C}_p}).$ Let $\pi_p$ be the geodesic from $0$ to $\overline{\mc{C}_p}$.
Indeed, starting from $0$ one can first travel to $\overline{\mc{C}_p}$ via $\pi_p$, then move to $\mc{C}_p$ with a cost of at most $k$, move freely on the circuit $\mc{C}_p$, and finally reach the endpoint $x \in \mc{C}_p$ that satisfies $|\tau_{\overline{\mc{C}_p}} - x | \leq k.$ Then Fact \ref{fact:3} implies 
$$
|T(0,x) - T(\pi_1)| \leq k
$$
and similarly 
$$
|T(0,x) - T(0, \overline{\mc{C}_p})| \leq k.
$$

By a similar argument, we have

\begin{align}
\label{eqn:pi2}
T(\pi_2) = T(\tau_{\overline{\mc{C}_p}}, \overline{\mc{C}_{\ell}}) + \triangle T_{p, \ell},
\end{align}

where $| \triangle T_{p, \ell}| \leq 2k $ holds with probability $1$.

Combining Equation (\ref{eqn:A}) and (\ref{eqn:pi2}), we obtain

\begin{align}
\label{eqn:f1}
\bb{E}[T(0, \mc{C}_{\ell}) | \mc{F}_p] = T(0, \overline{\mc{C}_p})+ \bb{E}'[ T( \tau_{\overline{\mc{C}_p}}, \mc{C}_{\ell})]+ R_3
\end{align}
where $|R_3| \leq 4k$.

Similarly,
\begin{align}
\label{eqn:f2}
\bb{E} [T(0,\mc{C}_{\ell}) | \mc{F}_{p-1}] = T(0, \overline{\mc{C}_{p-1}}) + \bb{E}'[T(\tau_{\overline{\mc{C}_{p-1}}}, \mc{C}_{\ell})]+R_4,
\end{align}

where $|R_4|\leq 4k$. 

Combining the above two Equations (\ref{eqn:f1}) and (\ref{eqn:f2}) yields the conclusion.

\end{proof}





\begin{remark}
As in \cite{kestenzhang}, by a little extra work one can write 
$\triangle_p = \overline{\triangle_p}+ R_p$, where $|R_p| \leq 8k$, such that a truncated version of the random variables $\overline{\triangle_p}$ are independent. We will not use this fact in the remaining proof.
\end{remark}

\begin{definition}
Let us define 
\begin{equation}
\label{def:ell}
n(p, \omega, \omega') = m(m(p, \omega)+1, \omega').
\end{equation}
Using the definition of $m(p)$ in (\ref{eqn:defm}), $n$ is the first geometric scale after $m(p)$ that contains an open circuit.
\end{definition}

\begin{lemma}    
\label{lem:2}
Let $n(p, \omega, \omega')$ be as defined as (\ref{def:ell}) above.
Then 
\begin{equation}
\label{eqn:2.24}
\triangle_p(\omega) = T(\mc{C}_{p-1}(\omega), \mc{C}_{p}(\omega))(\omega)+\bb{E}' T(\mc{C}_p(\omega), \mc{C}_{n(p, \omega, \omega')}(\omega'))(\omega') - \bb{E}' T(\mc{C}_{p-1}(\omega), \mc{C}_{n(p, \omega, \omega')} (\omega'))(\omega') + R'
\end{equation}
 where $|R'| \leq 8k$ with probability $1$.
\end{lemma}

\begin{proof}
The statement is similar to Lemma \ref{lem:1} but with different notation. 
The scale $n(p, \omega, \omega')$ is found with the following procedure. 
First, determine $m(p,\omega)$. Then by exploring $p_c$-open clusters, one can find the smallest $t \geq m(p, \omega) +1$ such that there is an open circuit surrounding $\mathbf{0}$ in $A(t)$ in configuration $\omega'$. 
The value of $t$ is $n(p, \omega, \omega')$, and the innermost open circuit in $A(n(p, \omega, \omega'))$ surrounding the origin in configuration $\omega'$ is $\mc{C}_{n(p, \omega, \omega')}(\omega').$ 
We have that 
\begin{align*}
n(p, \omega, \omega') \geq m(p, \omega) +1,
\end{align*}
which lets us see that
\begin{align*}
\mc{C}_p (\omega) = \mc{C}_{m(p, \omega)}(\omega) \subset A(m(p, \omega)) \subset \text{int}(\mc{C}_{n(p, \omega, \omega')}(\omega')).
\end{align*}
Then the same argument as for Lemma \ref{lem:1} leads to (\ref{eqn:2.24}).
\end{proof}

\begin{lemma}
\label{lem:3}
There are constants denoted by $C_i >0 $ such that for $p, q \geq 1$ we have that,

\begin{equation}
\label{eqn:2.28}
\bb{P} [ m(p) - p \geq t ] \leq e^{-C_5 t } \; \forall t, p \geq 0
\end{equation}

\begin{equation}
\label{eqn:2.29}
\bb{P} [| \triangle_p | \geq x ] \leq C_6 e^{- c \sqrt{x} } \; \text{ for } x \text{ large enough}
\end{equation}

\begin{equation}
\label{eqn:2.30}
\bb{P} [ \max_{0 \leq p \leq q} | \triangle_p | \geq \epsilon q^{1/2} ] \leq 2 C_6 q e^{- c_1 q^{1/4} } \epsilon^{\frac{1}{2}}, \forall \epsilon >0 
\end{equation}

\begin{equation}
\label{eqn:2.31}
\bb{E} [ \max_{ 0 \leq p \leq q } \triangle_p^2 ] \leq C_7 q
\end{equation}

\begin{equation}
\label{eqn:2.32}
C_8 q \leq \sum_{p=0}^q \bb{E} \triangle_p^2 \leq C_9 q
\end{equation}
 
\end{lemma}

\begin{proof}

We know that $m(p) - p \geq t $ occurs if and only if 
there is no blocking surface surrounding the origin in any of the 
annuli $A(p) , A(p+1), \cdots, A(p'-1)$, with $p' = p+t$. 

It is known (\cite{smythewierman} or \cite{kesten82} or \cite{grimmett89}) that there is a constant $C_5>0$ such that 

\begin{equation}
\bb{P} [ N_j := \text{ there is no blocking surface surrounding the origin in } A(j) ] \leq e^{-C_5}, j \geq 0. 
\end{equation}

The annuli $A(j)$ are disjoint, and therefore the events $N_j$, of no blocking surfaces surrounding the origin, are independent for distinct $j$.  So (\ref{eqn:2.28}) follows from this fact.

By (\ref{eqn:2.28}), we have the following for each fixed $\omega$,

\begin{equation}
\bb{P}' [n(p, \omega, \omega') \geq m(p, \omega)+1+t ] \leq e^{-C_5 t}. 
\end{equation}

Now we want to obtain (\ref{eqn:2.29}).

By (\ref{eqn:2.24}) of Lemma \ref{lem:2}, 

\begin{equation}
\label{eqn:2.35}
|\triangle_p (\omega) | \leq T(\mc{C}_{p-1}(\omega), \mc{C}_p(\omega))(\omega) + \bb{E}' T(\mc{C}_{p-1}(\omega), \mc{C}_{n(p, \omega, \omega')}(\omega'))(\omega') + 8k.
\end{equation}

It is easy to see $T(\mc{C}_{p-1}(\omega), \mc{C}_p(\omega)) \leq \mc{C} 2^p$, which is bounded.

Now, we estimate for fixed $\omega$ the tail probability 
\begin{equation}
\label{eqn:2.36}
\bb{P}' [ T ( \mc{C}_{p-1}(\omega), \mc{C}_{n(p, \omega,\omega')}(\omega')) > y].
\end{equation}

For the case when $p=0$ we shall interpret $S(2^{-1})$ be the origin $\mathbf{0}$, and let $A(-1):= S(1).$
For a square $S$ let $\partial S$ denote its topological boundary, let $\text{int}(S)$ denote its interior with the conditions $\partial S(2^{-1})= \mathbf{0}$ and $\text{int}(S(2^{-1})) = \emptyset$.

Given a fixed $\omega$, we are given a $m = m(p, \omega)$ and $\mc{C}_p (\omega) \subset A(m)$. For $r \geq n(p, \omega, \omega')+1$, any path on $\bb{S}_k$ from $\partial S(2^{p-1})$ to $\partial S(2^r)$ must intersect $\mc{C}_{p-1}(\omega)$ and $\mc{C}_{\ell(p, \omega, \omega')}.$

Therefore, Fact \ref{fact:3} implies 

\begin{equation}
\label{eqn:2.37}
T( \mc{C}_{p-1} (\omega), \mc{C}_{n(p, \omega, \omega')} (\omega'))(\omega') \leq T( \partial S(2^{p-1}),\partial S(2^r))(\omega') + 4k.
\end{equation}

So it follows that for $t=0,1, \cdots,$

\begin{align}
\bb{P}' [T( \mc{C}_{p-1}(\omega), \mc{C}_{n(p, \omega, \omega')}(\omega'))(\omega') \geq y] &  \\
\leq \bb{P}' [ \ell(p, \omega, \omega') \geq m(p, \omega) +1+t]& +\bb{P}'[ T(\partial S(2^{p-1}), \partial S(2^{m(p, \omega) +1 + t}))(\omega') \geq y ]  -4k \\
\leq e^{-C_5 t} + \bb{P}' [ T(\partial S(2^{p-1}), &\partial S(2^{m(p, \omega)+1+t})(\omega') \geq y ] -4k. \label{eqn:2.38}
\end{align}

To estimate the RHS of the above, let us define

$$
\kappa( j, k , \omega') = \text{ minimal number of $p_c$-closed edges in any path from } \partial S(2^j) \text{ to } \partial S( 2^k) \text{ in } \omega'
$$

$$
\rho ( j, k , \omega' ) = \text{ maximal no. of edge-disjoint closed dual circuits which surround } S(2^j) \text{ in } S(2^k) \setminus S(2^j) \text{ in } \omega'. 
$$

We may see that the number of closed edges $\kappa$ is equal to the number of dual surfaces $\rho$, 

\begin{align}
\label{eqn:closededgesdualsurfaces}
	\kappa( j, k, \omega') = \rho(j, k, \omega'). 
\end{align}

This is an example of the max-flow-min-cut theorem. For completeness we give a sketch of the proof in the Appendix.

We need to estimate the tail probability of $\rho$, and to do this we introduce events of $\Omega'$ given by

\begin{align}
	\label{eqn:defG}
	G(y) = G(y,j,k) = \{ \rho(j,k,\omega) \geq \lfloor y \rfloor \} = \{ \text{ There exist at least } \lfloor y \rfloor 
\end{align}
$$
\text{ disjoint closed dual circuits surrounding } S(2^j) \text{ in } S(2^k) \setminus S(2^j) \}.
$$

$\bb{P}(G(y))$ can be estimated using the BK inequality [\cite{BergKesten}], stated as follows.

\begin{theorem} (BK Inequality)
For events $G_1, \cdots, G_r \subset \Omega'$ that depend on finitely many variables $J(e) := I [t(e,\omega') \text{ is open}]$,

\begin{align}
	\label{eqn:Gs}
	\bb{P}' [ G_1 \square G_2, \cdots \square G_r ] \leq \prod_{i=1}^r \bb{P}' [ G_i ]
\end{align}

where $ G_1 \square G_2, \cdots \square G_r $ is the event that $G_1, \cdots, G_r$ occur disjointly.

\end{theorem} 

To study $\bb{P}(G(y))$, we apply the inequality with events given by
\begin{align*}
	G_i = G ( 2 C_{10} ( k- j) + 1)
\end{align*}

which is a decreasing event; its characteristic function is a decreasing function of $J(e)$. 
\cite{BergKesten} have shown Equation \ref{eqn:Gs} for this case. 

By definition of $G$ in Equation (\ref{eqn:defG}),

\begin{align}
	\label{eqn:rGs}
	G ( r[ 2 C_{10} (k-j) + 1]) \subset G( 2 C_{10} (k-j) + 1) \square \cdots \square G(2 C_{10} (k-j) +1),
	\end{align}

with $r$ events on the RHS above (disjoint occurrence). 

Let us take 

\begin{align}
	r = [ \frac{y}{4 C_{10} (k-j)+1}]
\end{align}

Therefore, the following holds,

\begin{align}
	\label{eqn:2.48}
	\bb{P}' [ \nexists \text{ a path } \gamma: \partial S(2^j) \rightarrow \partial S(2^k) \text{ with } \leq r[ 2C_{10} (k-j)+1] \text{ closed edges }] 
	\end{align}
	\begin{align*}
	&\leq \bb{P}' [ \rho \geq r[ 2C_{10} (k-j)+1] ] \\
	&\leq \bb{P}' [G(r[ 2C_{10} (k-j)+1] )] \\
	&\leq ( \bb{P}' [G(r[ 2C_{10} (k-j)+1] ) ] )^r \text{ by Equations (\ref{eqn:rGs}) and (\ref{eqn:Gs}) } \\
	& \leq 2 \cdot 2^{-C_{12} y/ (k-j)}.
	\end{align*}

If an event given by the LHS of Equation (\ref{eqn:2.48}) fails, then there exists a path $\pi: \partial S(2^j) \rightarrow \partial S(2^k)$ in $\omega'$ with at most
\begin{align*}
	s := \lfloor r [ 2 C_{10} (k-j) +1] \rfloor
	\end{align*}
closed edges. 
Now, $T( \partial S(2^j), \partial S(2^k))$ is dominated by $s$, and by our choice of $r$, is bounded by $\frac{y}{2} + 1$.

We take $j = p-1, k= m(p, \omega)+1+t$ to reach the estimate for Equation (\ref{eqn:2.36}). 
Combining Equation (\ref{eqn:2.38}) and (\ref{eqn:2.48}), for $t=0,1, \cdots$, 

\begin{align}
	\bb{P}' [ T(\mc{C}_{p-1}(\omega), \mc{C}_{n(p, \omega, \omega'} (\omega'))(\omega') \geq y ] \leq e^{-C_5 t} + 2 \cdot 2^{-C_{12} y/(m - p + t + 2)} + \bb{P}( s \geq y).
\end{align}

Taking $t = \lfloor \sqrt{y} \rfloor$, for constants $C_{15}, C_{16} \in (0, \infty)$ and for all $\omega \in \Omega, y \geq 0$, 

\begin{align}
	\label{eqn:2.52}
	\bb{P}' [ T (\mc{C}_{p-1}(\omega), \mc{C}_{n (p, \omega, \omega')} (\omega'))(\omega') \geq y] \leq C_{15} exp (-C_{16} \frac{y}{m-p+\sqrt{y}})
\end{align}

By the integration of Equation \ref{eqn:2.52} over $y$, an upper bound of the second term of the RHS of Equation (\ref{eqn:2.35}) is obtained,

\begin{align}
	\label{eqn:2.53}
\bb{E}' T(\mc{C}_{p-1}(\omega), \mc{C}_{n(p, \omega, \omega')}(\omega'))(\omega')  \leq	C_{17} [ m(p, \omega) - p +1]. 
\end{align}

Therefore, by Equation (\ref{eqn:2.35}), for $t=0, 1, \cdots, \lfloor x/2C_{17} \rfloor$, 

\begin{align}
	\label{eqn:2.54}
	\bb{P} [ | \triangle_p (\omega)| \geq x] \leq \bb{P} [m(p,\omega)- p \geq t]+ \bb{P}[ T(\mc{C}_{p-1}(\omega), \mc{C}_p (\omega))(\omega) \geq x/2, m(p, \omega)- p <t]. 
\end{align}

The 2nd term of the RHS above is at most
\begin{align}
	\bb{P} [T(\partial S(2^{p-1}), \partial S (2^{p+1}))(\omega) \geq x/2] \leq 2 \cdot 2^{-C_{12}x/2(t+1)}.
\end{align}

Therefore, Equations (\ref{eqn:2.54}) and (\ref{eqn:2.28}) let us conclude for $t \leq x/2C_{17}$,

\begin{align}
	\bb{P} [ |\triangle_p(\omega)| \geq x] \leq e^{-C_5 t} + 2 \cdot 2^{-C_{12}x/2(t+1)}.
	\end{align}

Taking $t = \lfloor \sqrt{x} \rfloor$, Equation (\ref{eqn:2.29}) of Lemma \ref{lem:1} follows. 

From (\ref{eqn:2.29}), (\ref{eqn:2.30}) and (\ref{eqn:2.31}) clearly follow, as well as the second inequality of (\ref{eqn:2.32}). We only need to show the first inequality of (\ref{eqn:2.32}).
  
  By Equations (\ref{eqn:2.24}) and (\ref{eqn:2.53}),  
  \begin{align*}
  	\triangle_p (\omega) \geq T(\mc{C}_{p-1}(\omega), \mc{C}_p(\omega)) (\omega) - C_{17} [ m(p, \omega) - p +1] - 8k.
  \end{align*}

Observe that a path which crosses $k$ closed dual circuits must have a passage time of at least $k$. So for $p+2 \leq q$,

$$
\bb{E} \triangle_p^2 \geq \bb{P} [\triangle_p \geq 1] \geq \bb{P} [ m(p, \omega)= p+1 \text{ and } T(\mc{C}_{p-1}(\omega), \mc{C}_p(\omega))(\omega) \geq 2C_{17}+1+8k]
$$
\begin{align}
	\geq \bb{P} [ \mc{C}_{p-1} (\omega) \subset A(p-1) \nexists \text{ open circuit in } A(p) \text{ but } \exists \text{ at least } (2C_{17}+1+8k) 
\end{align}
$$
\text{ edge-disjoint closed dual circuits surrounding } S(2^{p-1}) \subseteq A(p) \text{ and there is an open circuit in } A(p+1). ]
$$

Let 
$$
\mc{E}_1 := \{ \text{There exist open circuits surrounding } \mathbf{0} \text{ in } A(p-1) \text{ and } A(p+1) \},
$$

$$
\mc{E}_2 := \{ \text{ There does not exist an open circuit surrounding } \mathbf{0} \text{ in } A(p) \},
$$
and
$$
\mc{E}_3 := \{ \text{There exist at least } (2C_{17}+1+8k) \text{ edge-disjoint closed dual circuits surrounding } S(2^{p-1}) \text{ in } A(p) \}.
$$

By the independence of edges in $A(p-1), A(p), A(p+1)$ along with the Harris-FKG inequality,  

\begin{align}
	\label{eqn:2.57}
	\bb{E} \triangle_p^2 \geq \bb{P} (\mc{E}_1) \cdot \bb{P} (\mc{E}_2) \cdot \bb{P} (\mc{E}_3).
\end{align}

The Russo-Seymour-Welsh Theorem (Cor.\ref{cor:1}) implies that the probabilities on the RHS of Equation (\ref{eqn:2.57}) are bounded away from $0$. We can therefore see that (\ref{eqn:2.32}) follows.  

\end{proof}

\begin{lemma}
\label{lemm:4}
The following holds as $q \rightarrow \infty$, 
\begin{equation}
\label{lem:4}
\frac{T(\mathbf{0}, \mc{C}_{m(q)}) - \bb{E} T(\mathbf{0}, \mc{C}_{m(q)})}{[ \sum_{p=0}^q \bb{E} \triangle_{p,q}^2 ]^{1/2}} \rightarrow N(0,1) \text{ in distribution.}
\end{equation}
\end{lemma}

We prove Lemma \ref{lemm:4} by proving the three conditions of McLeish's Theorem [Theorem 2.3 \cite{McLeish}], recalled below.

\begin{theorem}
\label{thm:McLeish} (McLeish's Theorem)
Let $X_{n,i}$ be a martingale difference array that satisfies the following,
\begin{itemize}
\item $ \max_{i \leq k_n} |X_{n,i} | $ uniformly bounded in $\ell_2$-norm,
\item $\max_{i \leq k_n} |X_{n,i} |  \ra_p 0$
\item $\sum_i X_{n,i}^2 \ra_p 1$.
\end{itemize}
Then $S_n = \sum_{i=1}^{k_n} X_{n,i} \ra_{\omega} N(0,1).$ 
\end{theorem}

\begin{proof} (Lemma \ref{lem:4})

We may set 
$$
X_{p,q}  = \frac{ \triangle_{p,q}}{ [\sum_{p=0}^q \bb{E} \triangle_{p,q}^2 ]^{1/2}}
$$ 

which would let us express the left hand side of (\ref{lem:4}) as, 
\begin{equation}
\label{eqn:64}
\frac{T(\mathbf{0}, \mc{C}_{m(q)}) - \bb{E} T(\mathbf{0}, \mc{C}_{m(q)} )}{ [\sum_{p=0}^q \bb{E} \triangle_{p,q}^2 ]^{1/2}} = \sum_{p=0}^q X_{p,q}
\end{equation}

We now apply Theorem \ref{thm:McLeish} to (\ref{eqn:64}). 

By Lemma \ref{lem:3} Eq. (\ref{eqn:2.32}) we have 
$$
|X_{p,q}| \leq | \triangle_{p,q} | / [ C_8 q]^{1/2}
$$
 
 Conditions 1 and 2 of McLeish's Lemma follow directly from (\ref{eqn:2.30}) and (\ref{eqn:2.31}), since (\ref{eqn:2.30}) gives a tail bound for $\max_{i \leq k_n } | X_{n,i} | $ and (\ref{eqn:2.31}) gives the bound on the $\max_{i \leq k_n} |X_{n,i} |^2$.  

Now it remains to prove the last condition,
\begin{equation}
\sum_{p=0}^q X_{p,q}^2 \rightarrow 1 \text{ in probability. }
\end{equation}

The condition is equivalent to 
\begin{equation}
\label{eqn:2.60}
 \frac{1}{q} \sum_{p=0}^q [ \triangle_{p,q}^2 - \bb{E} \triangle_{p,q}^2 ] \rightarrow 0 \text{ in probability.}
\end{equation}

This is a weak law of large numbers type statement, and we prove it by bounding the variance of (\ref{eqn:2.60}). 


We denote by 

$$
\tilde{\triangle}_{p,q} = \triangle_{p,q} I [m(p) \leq p + \frac{3}{C_5} \log q].
$$

Then, 
\begin{align*}
\bb{P} \{ \triangle_{p,q} \neq \tilde{\triangle}_{p,q} \text{ for some } p \leq q \} &\leq \sum_{p=0}^q \bb{P} \{ m(p) - p \geq \frac{3}{C_5} \log q \} \\ 
&\leq (q+1) e^{-3 \log q} \text{ by (\ref{eqn:2.28})} \\
&\rightarrow 0.  
\end{align*}

Apply (\ref{eqn:2.29}) and (\ref{eqn:2.28}) and we have 
\begin{align}
\sum_{p=0}^q [\bb{E} \triangle_{p,q}^2 - \bb{E} \tilde{\triangle}_{p,q}^2] = \sum_{p=0}^q \bb{E} [ \triangle_{p,q}^2 I [m(p) - p \geq  \frac{3}{C_5} \log q ]] = o(q).
\end{align}


Now to show (\ref{eqn:2.60}) it suffices to prove the following,

\begin{equation}
\label{eqn:2.61}
\frac{1}{q} \sum_{p=0}^q [\tilde{\triangle}_{p,q}^2 - \bb{E} \tilde{\triangle}_{p,q}^2 ] \rightarrow 0 \text{ in probability.}
\end{equation}

We obtain (\ref{eqn:2.61}) by bounding the variance of the expression.  

Note that $\tilde{\triangle}_{p,q}$ and $\tilde{\triangle}_{r,q}$ are independent when 

$$
|p - r| \geq (\frac{3}{C_5}) \log q + 2
$$ holds. We also have the following uniform bound for $0 \leq p \leq q$,

\begin{align*}
Var(\tilde{\triangle}_{p,q}^2) &\leq \bb{E} [\tilde{\triangle}_{p,q}^4 ] \leq 4 \int_0^{\infty} x^3 \bb{P} ( |\triangle_{p,q} | \geq x ) dx \\
&\leq \int_0^{\infty} x^3 e^{-c_3 \sqrt{x}}dx  < \infty.
\end{align*}


This shows each of the variances of $\tilde{\triangle}_{p,q}^2$ are finite. We just need to count the total amount.

This yields the following, 



\begin{align*}
Var(\sum_{p=0}^q [\tilde{\triangle}_{p,q}^2 - \bb{E} \tilde{\triangle}_{p,q}^2]) \\
   &\leq 2 \sum_{p=0}^q \sum_{p \leq r \leq p+(3/C_5)\log q +2} [Var (\tilde{\triangle}_{p,q}^2) Var (\tilde{\triangle}_{r,q}^2) ]^{1/2}\\
   &= O( q \cdot \log q). 
\end{align*}

which shows (\ref{eqn:2.61}) holds. Therefore the lemma follows from McLeish's Theorem 2.3 [\cite{McLeish} 1974]. 
\end{proof}

\subsection{Main results}
Now, we can prove the main results.

We generalize the $n$ to be any real number, from our previous requirement for $n$ to be a power of $2$. Given 
\begin{equation}
\label{eqn:2.62}
2^{q-1} < n \leq 2^q,
\end{equation} define

\begin{equation}
\gamma_n = [\sum_{p=0}^q \bb{E} \triangle_{p,q}^2 ]^{1/2}.
\end{equation}

We must show that Equations (\ref{eqn:1.7}) and (\ref{eqn:1.8}) hold.

We define the half slab as follows,

\begin{definition}
The half slab $H_n$ is given by 
$$
\{ (x,y,z) \in \bb{S}_k, x \geq n, z \in \{0,\cdots,k\} \} 
$$
for some $n \geq 0$.
\end{definition}

\emph{Proof} of (\ref{eqn:1.7}), (\ref{eqn:1.8}).

From \ref{eqn:2.32} we see that $\gamma_n$ defined in above equation \ref{eqn:2.63} must satisfy (\ref{eqn:1.7}).

We note that $\mc{C}_{m(q)}$ surrounds the origin and it lies outside of $S(2^q)$, and therefore outside $S(n)$ given that $n$ is a dyadic scale or (\ref{eqn:2.62}) is satisfied. So $\mc{C}_{m(q)}$ must contain points in the half slab $H_n$ and therefore, 
\begin{equation}
\label{eqn:2.64}
0 \leq b_{0,n} \leq T(0, \mc{C}_{m(q)}),
\end{equation}

since the passage time from the origin to a point $v$ is the same for all $v \in \mc{C}_{m(q)}$.

If for some $k$ the following holds,

\begin{equation}
\label{eqn:2.65}
\mc{C}_{m(q-k)} \subset S(2^{q-1}) \subset S(n),
\end{equation}

any path from the origin to $H_n$ must intersect $\mc{C}_{m(q-k)}$ and the following holds,

\begin{equation}
b_{0,n} \geq T( \mathbf{0}, \mc{C}_{m(q-k)}) = T( \mathbf{0}, \mc{C}_{m(q)}) - T(\mc{C}_{m(q-k)}, \mc{C}_{m(q)}).
\end{equation}

When $m(q) \leq q + t$ then the following holds,
\begin{equation}
\label{eqn:2.67}
T( \mc{C}_{m(q-k)}, \mc{C}_{m(q)}) \leq T(\partial S(2^{q-k}), \partial S(2^{q+t+1})),
\end{equation}
as in (\ref{eqn:2.37}).

Equations (\ref{eqn:2.64})-(\ref{eqn:2.67}) show that when (\ref{eqn:2.62}) holds, then for all $x \geq 0$ and $k \leq q, t \geq 0$,

\begin{align}
\bb{P} [| b_{0,n} - T( \mathbf{0}, \mc{C}_{m(q)}) | \geq x]& \nonumber \\
        &\leq \bb{P} [ m(q-k) \geq q-1 ] +\bb{P} [ m(q) \geq q +t] + \bb{P} [ T(\partial S(2^{q-k}), \partial S(2^{q+t+1}) \geq x ] \nonumber \\
    & \leq e^{-C_5(k-1)} + e^{-C_5 t}+2 \cdot 2^{-C_{12} x / (k+t+1)}. \label{eqn:2.68}
\end{align}

Take $k=t= \lfloor \sqrt{x} \rfloor $ satisfying $ t \leq q$. Therefore, for $C_{21}< \infty, \; x \leq q^2$, the following holds,

\begin{equation}
\label{eqn:2.69}
\bb{P} [ |b_{0,n} - T( \mathbf{0}, \mc{C}_{m(q)}) | \geq x ] \leq C_{21} e^{- c_0 \sqrt{x}}.
\end{equation}

This holds even for $x \geq q^2$.

From (\ref{eqn:2.64}), by (\ref{eqn:2.28}) and (\ref{eqn:2.50}),

\begin{align*}
\bb{P} [ | b_{0,n} - T( \mathbf{0}, \mc{C}_{m(q)} ) | \geq x ] \leq& \bb{P} [ T( \mathbf{0}, \mc{C}_{m(q)}) \geq x ] \\
\leq& \bb{P} [ m(q) \geq q+t ] + \bb{P} [ T(\mathbf{0}, \partial S(2^{q+t})) \geq x ] \\
\leq& e^{-C_5 t} + 2 \cdot 2^{-C_{12} x / (q+t+1)}.
\end{align*}

From (\ref{eqn:2.69}), for $q$ chosen s.t. (\ref{eqn:2.62}) holds, and as $n \rightarrow \infty$,

\begin{equation}
\frac{b_{0,n} - \bb{E} T( \mathbf{0}, \mc{C}_{m(q)}) }{\gamma_n} \rightarrow 0 \text{ in probability.}
\end{equation}

Because the following holds,
$$
\gamma_n = \gamma_{2^q} = [ \sum_{p=0}^q \bb{E} \triangle_{p,q}^2 ]^{1/2},
$$
by (\ref{eqn:2.62}), and by Lemma \ref{lemm:4} (\ref{lem:4}), we may conclude the following,

$$
\frac{ b_{0,n} - \bb{E} T( \mathbf{0}, \mc{C}_{m(q)}) }{\gamma_n} \rightarrow N(0,1) \text{ in distribution. }
$$

Finally, to prove (\ref{eqn:1.8}), it is clear from (\ref{eqn:2.69}) that 
$$
\bb{E} b_{0,n} - \bb{E} T(\mathbf{0}, \mc{C}_{m(q)}) \text{ is bounded.}
$$

Let us define the following variables,

\begin{equation}
s_n = T( \mathbf{0}, \partial S(n)).
\end{equation}

Any path from $\mathbf{0} \rightarrow H_n$ must intersect $\partial S(n)$, so that 
\begin{equation}
s_n \leq b_{0,n}.
\end{equation}

We also have the following whenever (\ref{eqn:2.65}) holds,
\begin{equation}
\label{eqn:2.73}
T(\mathbf{0}, \partial S(2^{q-k})) \leq s_n.
\end{equation}

So with the same proof technique as in the last proof,

\begin{equation}
\label{eqn:2.74}
\bb{P} [ |s_n - T(\mathbf{0}, \mc{C}_{m(q)})| \geq x ] \leq C_{21} e^{-c_0 \sqrt{x} }
\end{equation}

and that 

\begin{equation}
\frac{s_n - \bb{E} s_n}{\gamma_n} \rightarrow N(0,1)
\end{equation}

in distribution.

From (\ref{eqn:2.74}) and (\ref{eqn:2.69}), 

\begin{equation}
\bb{E} S_n = \bb{E} b_{0,n} + O(1),
\end{equation}
and therefore we conclude (\ref{eqn:1.8}).

Now we prove the main result,

$$
\frac{T(\mathbf{0}, nu) - \bb{E} T(\mathbf{0},nu)}{\sqrt{2} \gamma_n} \rightarrow N(0,1)
$$
in distribution where $N(0,1)$ is a standard normal variable with mean $0$ and variance $1$.

\begin{equation}
\bb{E} b_{0,n} - \bb{E} s_n = O(1).
\end{equation}

Similarly, one can use (\ref{eqn:2.83}) below to show that
\begin{equation}
\bb{E}T(0, nu)- 2\bb{E}c_n = O(1)
\end{equation}
 for every unit vector u.

Let $u = (u_1, u_2)$ be a unit vector with a fixed position. Let $0 \leq u_2 \leq u_1 \leq 1$, without loss of generality. Therefore, $u_1 \geq 2^{-1/2}.$

Let the following hold for scale we denote by $r$,

\begin{equation}
2^{r-1} < \frac{1}{2} n u_1 \leq 2^r
\end{equation}

which combined with (\ref{eqn:2.62}) and the fact that $u_1 \geq 2^{-1/2}$ gives us that

\begin{equation}
\label{eqn:2.79}
q-3 \leq r \leq q.
\end{equation}

Consider the two squares denoted as

$$
S' = S(2^{r-1})
$$

and 

$$
S'' = nu + S(2^{r-1}).
$$ 

The squares $S', S''$ are disjoint and $\mathbf{0} \in S'$ and $nu \in S''$.

Therefore, a path from the origin to $nu$ must contain the piece from $\mathbf{0}$ to the first intersection with $\partial S'$ and the piece of its last intersection with $\partial S''$ to $nu$. 

So the following must hold,

\begin{equation}
\label{eqn:2.80}
T(\mathbf{0}, nu) \geq T(\mathbf{0}, \partial S') + T(nu, \partial S'').
\end{equation}

Now we wish to obtain an estimate in the other direction, and we consider the annuli $A(p), \cdots, A(p+1),\cdots$ such that $p \geq q+2$.

We have for each $p$,

\begin{equation}
\label{eqn:2.81}
S' \cup S'' \subset S(2^p)
\end{equation}

since $n \leq 2^q$ and $|u| =1$.

Recall the definition of $m(q+2),$ given by $\text{inf}[p \geq q+2: \exists \text{ dual blocking surface surrounding the origin in } A(p)]$.

Let $\mc{C} := \mc{C}_{m(q+2)}.$

By (\ref{eqn:2.81}), $\mc{C}$ must surround both $S', S''$ and therefore also must surround $\mathbf{0}$ and $nu$. Now, we connect $\mathbf{0}$ and $nu$ to $\mc{C},$ along an arc that lies on $\mc{C}$.

The following holds,

$$
\partial S (2^{m(q+2)+1}) \subset \text{int}(nu + S(2^{m(q+1)+2}))
$$
and leads to the following,

\begin{align*}
T( \mathbf{0}, nu) &\leq T(\mathbf{0}, \mc{C}) + T(nu, \mc{C}) \\
&\leq T(\mathbf{0}, \partial S(2^{m(q+2)+1})) + T(nu, nu + \partial S(2^{m(q+2)+2})). 
\end{align*}

By (\ref{eqn:2.80}) this gives,

\begin{align*}
\bb{P} [ | T(\mathbf{0},nu) - T(\mathbf{0}, \partial S') - T(nu, \partial S'') | \geq 2x] &\leq \bb{P} [ |  T(\mathbf{0}, \partial S') - T(\mathbf{0}, \partial S(2^{m(q+2)+1})) | \geq x ]  \\
+& \bb{P} [ | T(nu, \partial S '') - T(nu,nu+ \partial S ( 2^{m(q+2)+2})) | \geq x ] \\
=& \bb{P} [ |T(\mathbf{0}, \partial S(2^{r-1}) - T(\mathbf{0}, \partial S(2^{m(q+2)+1})) | \geq x ] \\
&+ \bb{P} [ | T(\mathbf{0}, \partial S(2^{r-1})) - T(\mathbf{0}, \partial S(2^{m(q+2)+2})) | \geq x] \\
&\leq 2 \bb{P} [m(q+2) \geq q+2+t] + 2 \bb{P} [m(q-k) \geq q-4] \\
 &+ \bb{P} [T(\partial S(2^{q-k}), \partial S(2^{q+3+t})) \geq x] \\
 &+ \bb{P} [T(\partial S(2^{q-k}), \partial S(2^{q+4+t} )) \geq x] \\
 &\leq 2e^{-C_5 t} + 2 e^{-C_5(k-4)} + 4 \cdot 2^{-C_{12} x /(k+t+4)}. 
\end{align*}

The above follows from (\ref{eqn:2.28}), (\ref{eqn:2.50}), (\ref{eqn:2.79}) and translation invariance.

Take $t = k = \lfloor \sqrt{x} \rfloor$, and this yields,

\begin{equation}
\label{eqn:2.83}
\bb{P} [ | T(\mathbf{0}, nu) - T(\mathbf{0}, \partial S') - T(nu, \partial S'') | \geq 2x ] \leq C_{22} e^{-c_0 \sqrt{x}}.
\end{equation}

Since $S'$ and $S''$ are disjoint we may recall that $T(\mathbf{0}, \partial S')$ and $T(nu, \partial S'')$ are independent, both with the distribution $s_{2^{r-1}} = T(\mathbf{0}, \partial S')$.

It therefore follows that,

\begin{equation}
\label{eqn:2.84}
\frac{1}{\sqrt{q}} [T(\mathbf{0}, nu) - T(\mathbf{0}, \partial S') - T(nu, \partial S'')] \rightarrow 0 \text{ in probability.}
\end{equation}

This lets us conclude (by (\ref{eqn:2.75})), 

\begin{equation}
\label{eqn:2.75}
\frac{T(\mathbf{0}, nu) - 2 \bb{E} s_{2^{r-1}}}{\sqrt{2} \gamma_{2^{r-1}} } \rightarrow N(0,1) \text{ in distribution.}
\end{equation}

Now we must show the following:
\begin{equation}
\label{eqn:2.86}
\frac{\gamma_{2^{r-1}}}{\gamma_n} \rightarrow 1,
\end{equation}

We must additionally show the following:

\begin{equation}
\label{eqn:2.87}
\bb{E} T(\mathbf{0}, nu) - 2 \bb{E} c_{2^{r-1}} = O(1).
\end{equation}

We state the following fact,

\begin{fact}
\label{fact:last}
For some $k$ fixed, if the following holds,
\begin{align*}
2^{q-k} < \tilde{n} \leq n \leq 2^q,
\end{align*}
then 
\begin{align*}
s_{2^{q-k}} \leq s_{\tilde{n}} \leq s_n \leq s_{2^q},
\end{align*}
\end{fact}

It is clear that (\ref{eqn:2.86}) follows from Fact \ref{fact:last}.

Therefore,

$$
\frac{1}{\sqrt{q}} [ c_{2^q} - c_{2^{q-k}}] \rightarrow_p 0 \text{ as } q \rightarrow \infty 
$$

from a similar argument as for (\ref{eqn:2.68}) and (\ref{eqn:2.69}).

Combined with (\ref{eqn:2.75}), for $n=2^q$ and $n=2^{q-k}$, this lets us conclude the error term satisfies,

$$
\frac{1}{\sqrt{q}} [ \bb{E} c_{2^q} - \bb{E} c_{2^{q-k}}] \rightarrow 0 \text{ and } \frac{ \gamma_n}{\gamma_{\tilde{n}}} \rightarrow 1.
$$

This is a special case of (\ref{eqn:2.86}) because of (\ref{eqn:2.79}), and (\ref{eqn:2.87}) follows immediately from (\ref{eqn:2.83}).

\section{Appendix}

\begin{proof} (Proof of (\ref{eqn:closededgesdualsurfaces}) )
It is clear that $\kappa \leq \rho$. The number of closed edges is less than number of closed dual circuits because each of the closed surfaces must give at least one closed edge. 
For the other direction, look at the open cluster that contains the closed dual circuit, and if it reaches the outer cluster then $\rho \leq \kappa$ is trivial since $\kappa = zero$ so it is possible to go from $2^j$ to $2^k$ with $0$ closed edges.
If this is not true then the open cluster must end at some endpoint with a fixed radius, and it is possible to find a closed dual surface which surrounds the closed edges. 
\end{proof}

\newpage

\bibliography{cltslabsThesis.bib}
\bibliographystyle{plain}

\end{document}